\RequirePackage{tikz}
\documentclass[sn-mathphys]{sn-jnl}

\usepackage{natbib}

\jyear{2021}%

\theoremstyle{thmstyleone}%
\newtheorem{theorem}{Theorem}

\newtheorem{problem}{Problem}[section]
\newtheorem{proposition}[theorem]{Proposition}%

\theoremstyle{thmstyletwo}%
\newtheorem{example}{Example}%

\theoremstyle{thmstylethree}%

\begin{document}

\title[Forbidden Subgraphs of Co-primes Graphs of Finite Groups]{Forbidden Subgraphs of Co-primes Graphs of Finite Groups}

\author*[1]{\fnm{Swathi} \sur{V V}}\email{swathi\_p180061ma@nitc.ac.in}

\author[1]{\fnm{Sunitha} \sur{M S}}\email{sunitha@nitc.ac.in}
\affil*[1]{\orgdiv{Department of Mathematics}, \orgname{National Institute of Technology Calicut}, \orgaddress{ \city{Calicut}, \postcode{673601}, \state{Kerala}, \country{India}}}

\abstract{For a finite group $G$ the co-prime graph $\Gamma(G)$ is defined as a graph with vertex set $G$ in which two distinct vertices $x$ and $y$ are adjacent if and only if $gcd(o(x),o(y))=1$ where $o(x)$ and $o(y)$ denote the orders of the elements $x$ and $y$ respectively. In this paper we find properties of groups whose co-prime graphs forbid graphs such as $C_4,K_{1,3},P_4$ and asteroidal triples.\\
  }



\maketitle

\section{Introduction}\label{sec1}
For a finite group $G$, several graphs can be defined using different group properties. Power graphs, Enhanced power graphs, commuting graphs, Intersection graphs and co-primes graphs are some of the examples of graphs defined on groups. There are many useful applications for graphs defined on groups and are related to automata theory \cite{kelarev2003graph}. 

Ma et.al \cite{ma2014coprime} introduced the co-prime graph $\Gamma(G)$ of a finite group $G$ in the year 2014. They defined the co-prime graph as a simple undirected graph with vertex set $G$ and the edge set consists of the unordered pair of vertices $(x,y)$ such that $gcd(o(x),o(y))=1$, where $o(x)$ and $o(y)$ denote the orders of the elements $x$ and $y$ respectively. The authors studied some graph theoretic properties such as diameter, clique number, planarity and automorphism group of co-prime graphs. The number of edges in co-prime graphs of cyclic groups and dihedral groups are computed in \cite{shelash2021co}. In \cite{dorbidi2016note} the authors proved that the clique number and chromatic number of co-prime are same for any finite group $G$. Also they classified all finite groups whose co-prime graphs are complete $r$-partite or planar. The authors in \cite{juliana2020coprime} studied co-prime graphs of cyclic groups. The clique numbers and chromatic numbers of co-prime graphs of dihedral groups are determined in \cite{syarifudin2021clique}. The co-prime graph of generalized quaternion groups are studied in \cite{nurhabibah2021some}. Refer \cite{hao2022notes} and \cite{saini2021co} for some related graphs.

 A forbidden graph characterization can be used to describe numerous significant families of graphs. Graph theorists have been searching for such characterization since in 1930, Kuratowski characterized planar graphs as those graphs which does not contain induced subgraphs isomorphic to $K_5$ or $K_{3,3}$. Cographs and split graphs are some of the important graph classes which can be defined using forbidden induced subgraphs. Forbidden subgraphs of power graphs of groups have been
studied by Doostabadi et al. \cite{doostabadi2014power} and Cameron et al. \cite{manna2020forbidden}. 

In this paper we classify finite groups whose co-prime graphs are AT-free, $C_4$-free, Claw-free, co-graphs and split-graphs.

\section{Preliminaries}

A graph $\Gamma$ is a pair $(V,E)$ of two sets, where $V$ is a finite nonempty set of objects called vertices and $E$ is a set of 2-element subsets of $V$ called edges. Two vertices $x$ and $y$ are said to be adjacent in $\Gamma$ if $\{x,y\}$ is an edge of $\Gamma$. We denote $x\sim y$ if the vertices $x$ and $y$ are adjacent in $\Gamma$. The degree of a vertex $x$ is the number of vertices adjacent to $x$, which is denoted by $deg(x)$. The set of all vertices adjacent to $x$ is called the neighborhood of $x$, which is denoted by $N(x)$.

Let $\mathcal{G}$ be a family of graphs. $\mathcal{G}$ is called $H$-free if no $G\in \mathcal{G}$ contains $H$ as an induced subgraph, also $H$ is called the forbidden subgraph of $\mathcal{G}$.
 A graph is called a co-graph if it is $P_4$-free. A split graph is a graph in which the vertices can be partitioned into an independent set  and a clique. A graph is said to be Claw-free if it is $K_{1,3}$-free. An Asteroidal Triple or briefly AT in a graph is an independent set of three vertices such that there is a path between each pair of those three vertices which does not contain any neighbours of the third. The graph given in Figure 1 is a cograph and a split graph that is not claw-free and not AT-free. The subgraph induced by the vertices $u,v,w$ and $r$ is a claw, and the vertices $u,w$ and $z$ form an AT.

 The order of a group $G$ is the number of elements in $G$, which is denoted by $\lvert G\rvert $. The identity element in $G$ is denoted by $e$. The order of an element $g\in G$ is the smallest positive integer $n$ such that $g^n=e$, which is denoted by $o(g)$.
We denote set of all prime divisors of $\lvert G\rvert$ by $\pi(G)$ and that of $o(g)$ by $\pi(g)$. The center of $G$ is the set of all elements in $G$ which commute with every other element in $G$, and is denoted $Z(G)$. An EPPO-group is a group in which every element has some prime power order.

\begin{picture}(100,200)(-100,-70)
\put(0,0){\circle*{5.5}}
\put(0,50){\circle*{5.5}}
\put(0,100){\circle*{5.5}}
\put(50,0){\circle*{5.5}}
\put(100,0){\circle*{5.5}}
\put(50,100){\circle*{5.5}}
\put(100,100){\circle*{5.5}}
\put(100,50){\circle*{5.5}}
\put(0,0){\line(0,0){100}}
\put(0,0){\line(1,0){100}}
\put(0,100){\line(1,0){100}}
\put(100,0){\line(0,0){100}}
\put(50,0){\line(0,0){100}}
\put(0,50){\line(1,0){100}}
\put(0,50){\line(1,-1){50}}
\put(50,100){\line(1,-1){50}}
\put(0,50){\line(1,1){50}}
\put(50,0){\line(1,1){50}}
\put(0,-10){\makebox(0,0)[cc]{$z$}}
\put(50,-10){\makebox(0,0)[cc]{$r$}}
\put(100,-10){\makebox(0,0)[cc]{$s$}}
\put(0,110){\makebox(0,0)[cc]{$u$}}
\put(50,110){\makebox(0,0)[cc]{$v$}}
\put(100,110){\makebox(0,0)[cc]{$w$}}
\put(-10,50){\makebox(0,0)[cc]{$x$}}
\put(110,50){\makebox(0,0)[cc]{$y$}}
\put(50,-40){\makebox(0,0)[cc]{Figure 1}}

\end{picture}

\section{Forbidden subgraphs}

In this section we characterize all groups whose co-prime graphs are $C_4$-free, claw-free, cographs and split graphs. We begin with a proposition which we use throughout this paper.

\begin{proposition}
Let $g_1$ and $g_2$ be two elements of a group $G$. Then the following holds,
\renewcommand{\labelenumi}{(\alph{enumi})}
\begin{enumerate}
    \item $g_1\sim g_2$ in $\Gamma(G)$ if and only if $\pi(g_1)\cap \pi(g_2)=\phi$.
    \item $\pi(g_1)\subseteq \pi(g_2)$ if and only if $N(g_2)\subseteq N(g_1)$.
\end{enumerate}

\end{proposition}
\begin{proof}
(a) The statement holds from the definition of the co-prime graph.\\
(b) Suppose $\pi(g_1)\subseteq \pi(g_2)$. If $x\in N(g_2)$ then $\pi(g_2)\cap \pi(x)=\phi$, which implies $\pi(g_1)\cap \pi(x)=\phi$, and hence $x\in N(g_1)$.

Conversely suppose $N(g_2)\subseteq N(g_1)$ and let $p\in \pi(g_1)$. If $p\notin \pi(g_2)$, then $g_2\sim x$ for some $x\in G$ with $o(x)=p$. Since $N(g_2)\subseteq N(g_1)$, $x\in N(g_1)$, which is not possible.
\end{proof}

Next we prove that co-prime graphs of finite groups form a universal graph, which gives some context to the results about groups defined by forbidden subgraphs.
\begin{theorem}
Every finite graph is an induced subgraph of the co-prime graph of some group.
\end{theorem}
\begin{proof}
Let $H=(V,E)$ be a graph with $\lvert V\rvert=n$ and $\lvert E\rvert=m$. \\
Supppose $W=\{u_1,u_2,...,u_r\}$ be the set of all vertices in $H$ of degree $n-1$ each.\\
Let $s=\frac{n(n-1)}{2}-m$ and $l=s+r$.\\ 
\underline{\textit{Claim}}: $H$ is an induced subgraph of $\Gamma(\mathbb{Z}_k)$ if $\lvert \pi(\mathbb{Z}_k)\rvert=l$. \\
Let $A=\{\alpha_1,\alpha_2,...,\alpha_s\}$ be the set of all 2-element subsets of $V$ which are not edges of $H$.\\
Let $P=\{p_1,p_2,...,p_s\}$ be the set of first $s$ primes and $Q=\{p_{s+1},p_{s+2},...,p_{s+r}\}$ be the set of next $r$ primes.\\
Define bijections $f:A\longrightarrow P$ and $g:W\longrightarrow Q$ by $f(\alpha_i)=p_i$ and $g(u_j)=p_{s+j}$ respectively.\\
Now, for each $v\in V\setminus W$, let $P_v=\{f(\alpha_i): v\in \alpha_i\}$.\\
Consider the cyclic group $\mathbb{Z}_k$ such that $k=p_1p_2...p_sp_{s+1}...p_{s+r}$.\\
For each $v\in V\setminus W$, select an element $x_v$ from $\mathbb{Z}_k$, distinct for distinct $v$, such that
$o(x_v)=\prod_{p_i\in P_v}^{}p_i$, and for each $u_j\in W$, select $x_{u_j}$ from $\mathbb{Z}_k$, distinct for distinct $u_j$ such that $o(x_{u_j})=p_{s+j}$.\\
Let $U_1=\{x_v: v\in V\setminus W\}$ and $U_2=\{x_{u_j}:u_j\in W\}$.\\
We prove that subgraph induced by the vertices in $U_1\cup U_2$ is isomorphic to $H$. Define $h:V\longrightarrow U_1\cup U_2$ by $h(v)=x_v$ if $v\in V\setminus W$ and $h(u_j)=x_{u_j}$ if $u_j\in W$. If $a$ and $b$ are two non-adjacent vertices in $G$ then $\{a,b\}=\alpha_i$ for some $\alpha_i\in A$, which implies $p_i\in P_a$ and $p_i\in P_b$, therefore, $p_i\mid o(x_a)$ and $p_i\mid o(x_b)$. Hence $x_a$ and $x_b$ are not adjacent in $\Gamma(U_1\cup U_2)$ and hence $h$ is an isomorphism.

\end{proof}

The following theorem gives a characterization to the co-prime graphs which are $C_4$-free.

\begin{theorem}
Let $G$ be a finite group. $\Gamma(G)$ is $C_4$-free if and only if $G$ is a nilpotent group of order $p^n$ or $2p^n$, $p$ is a prime.
\end{theorem}
\begin{proof}
If $G$ is a $p$-group, then $\Gamma(G)$ is a star graph, which is $C_4$-free.
Let $G$ be a nilpotent group of order $2p^n$ where $p$ is an odd prime. If $\Gamma(G)$ contains an induced $C_4$, say $g_1\sim g_2\sim g_3\sim g_4$, then the orders of $g_1$ and $g_3$ must be some powers of $p$ and orders of $g_2$ and $g_4$ must be 2. which is not possible since $G$ contains unique element of order 2 as it is a nilpotent group. 

Now, suppose $|G|$ has two odd prime divisors, say $p$ and $q$. Take $x_1,x_2\in G$ having order $p$ and $y_1,y_2$ having order $q$, then $x_1\sim y_1\sim x_2\sim y_2$ is an induced $C_4$. 

Suppose $G$ is not a nilpotent group of order $2p^n$ where $p$ is an odd prime and $n$ is a positive integer. Since the sylow-$p$ subgroup of $G$ is unique, the number of sylow-2 subgroups should be $p-1$, inorder to get rid of the nilpotency. Now, two elements of order 2 together with two elements of order $p$ will form a $C_4$.
\end{proof}

The next theorem characterises the co-prime graphs which are claw-free.
\begin{theorem}
Let $G$ be a finite group. $\Gamma(G)$ is Claw-free if and only if $\lvert G\rvert\leq3$.
\end{theorem}
\begin{proof}
If $p\mid \lvert G\rvert$ where $p$ is a prime greater than 3, then there are at least three elements each of order $p$, and any three of these elements together with the identity form a claw.

Suppose $\lvert G\rvert=2^i3^j,\, i>1\, \text{or}\, j>1$. If $i>1$ then three elements each of order a power of 2, or if $j>1$ then three elements each of order a power of 3, together with the identity induce a claw.

If $\lvert G\rvert=6$, then elements of orders 6, 2 and identity in $\mathbb{Z}_6$ and three elements of order 2 and identity in $S_3$ induce claws.

Clearly $\Gamma(G)$ is claw-free if $\lvert G\rvert\leq 3$.
\end{proof} 

Using similar proofs we can show that $\Gamma(G)$ is $K_{1,4}$-free if and only if $\lvert G\rvert\leq 4$.

Let $G$ be a finite group of order $n$. The \textit{Gruenberg-Kegel graph}\cite{cameron2022finite} or \textit{prime graph} of $G$ is a graph with vertex set $\pi(n)$ in which two distinct vertices $p_1$ and $p_2$ are adjacent if and only if $G$ contains an element of order $p_1p_2$.

A graph forbidding path of order 4 is called a cograph. Cographs have many key properties such as they form the smallest class of graphs containing the 1-vertex graph and closed under disjoint union and complimentation.
\begin{proposition}
Let $G$ be a finite group. If $\Gamma(G)$ is a \textit{cograph} then \textit{Gruenberg-kegel graph} of $G$ is $P_3$-free.

\end{proposition}
\begin{proof}
Suppose that \textit{Gruenberg-Kegel graph} of $G$ contains an induced $P_3$, say, $p_1,p_2,p_3$. Then $G$ contains elements of order $p_1p_2$ and $p_2p_3$ and does not contain an element of order $p_1p_3$. Take $x,y,z,w\in G$ such that $o(x)=p_1,\, o(y)=p_3,\, o(z)=p_1p_2$ and $o(w)=p_2p_3$. Then $z\sim y\sim x\sim w$ is a $P_4$ in $\Gamma(G)$.
\end{proof}
The converse of Proposition 2 need not be true as the co-prime graph of $\mathbb{Z}_{30}$ is not a cograph but its \textit{Gruenberg-Kegal graph} is $K_3$ which is $P_3$-free.
\begin{theorem}
Let $G$ be a finite nilpotent group. Then, $\Gamma(G)$ is a cograph if and only if $\lvert \pi(G)\rvert<3$.
\end{theorem}
\begin{proof}
Suppose that $\Gamma(G)$ contains an induced $P_4$, say $x,y,z,w$. Then by Proposition 1, $\pi(x)\cap \pi(y)=\phi,\, \pi(y)\cap\pi(z)=\phi, \, \pi(z)\cap \pi(w)=\phi,\, \pi(x)\cap\pi(z)\neq \phi, \, \pi(y)\cap \pi(w)\neq \phi,$ and $\pi(x)\cap\pi(w)\neq \phi$. Let $p_1\in \pi(x)\cap\pi(z),\, p_2\in \pi(y)\cap \pi(w)$ and $p_3\in\pi(x)\cap\pi(w) $, then $p_1\neq p_2\neq p_3$ since $\pi(x)\cap \pi(y)=\phi$, and hence $\lvert \pi(G)\rvert\geq3$.

Conversely, suppose $\lvert \pi(G)\rvert\geq3$ and let $p_1,p_2,p_3\in \pi(G)$. Since $G$ is nilpotent, there exist elements $x,y\in G$ such that $o(x)=p_1p_2$ and $o(y)=p_2p_3$. Let $z,w\in G$ such that $o(z)=p_1$ and $o(w)=p_3$. Then $x\sim w\sim z\sim y$ is an induced $P_4$ in $\Gamma(G)$.
\end{proof}
Using Proposition 2 and Theorem 3.4, we can generalize the result as follows.
\begin{theorem}
Let $G$ be a finite group. Then $\Gamma(G)$ is a cograph if and only if $G$ contains no elements $g_1$ and $g_2$ with orders $p_1p_2$ and $p_2p_3$ respectively where $p_1,p_2$ and $p_3$ are primes and $p_2\neq p_3$.
\end{theorem}
\begin{theorem}
For a finite group $G$, $\Gamma(G)$ is a split graph if and only if $G$ is a nilpotent group of order $p^n$ or $2p^n$ where $p$ is a prime.
\end{theorem}
\begin{proof}
We prove this using the characterization of split graphs which states that: a graph is a split graph if and only if it has no induced subgraphs isomorphic to $C_4$, $C_5$ and $2K_2$.

First we show that if $\Gamma(G)$ contains an induced $2K_2 $, then $\lvert \pi(G)\rvert\geq 4$. Suppose $x_1,x_2$ and $y_1,y_2$ forms $2K_2$ in $\Gamma(G)$. Then there exists primes $p_1,p_2,p_3$ and $p_4$ in $\pi(G)$ such that $p_1\in \pi(x_1)\cap \pi(y_1),\, p_2\in \pi(x_1)\cap \pi(y_2),\, p_3\in \pi(x_2)\cap \pi(y_1)$ and $p_4\in \pi(x_2)\cap \pi(y_2)$ and $p_1\neq p_2\neq p_3\neq p_4$ since $\pi(x_1)\cap\pi(x_2)=\phi$ and $\pi(y_1)\cap\pi(y_2)=\phi$. Similarly we can show that if $\Gamma(G)$ contains an induced $C_5$, then $\lvert \pi(G)\rvert\geq 5$.

Hence using Theorem 3.2, we conclude that $\Gamma(G)$ is split if and only if $G$ is a nilpotent group of order $p^n$ or $2p^n$ for some prime $p$.
\end{proof}
\section{Asteroidal triples}
In this section we investigate groups whose co-prime graphs does not contain asteroidal triples.
\begin{proposition}
Let $G$ be a finite group such that $\Gamma(G)$ contains an AT, then $\lvert\pi(G)\rvert\geq 3$.
\end{proposition}
\begin{proof}
Let the vertices $x_1,x_2$ and $x_3$ form an AT in $\Gamma(G)$. If $\pi(x_1)\subseteq \pi(x_2)$, then by Proposition 1, $N(x_2)\subseteq N(x_1)$ and hence any path between $x_2$ and $x_3$ will contain a neighbor of $x_1$, which is not possible. Hence, $\pi(x_i)\nsubseteq \pi(x_j)$, $i,j=1,2,3$ and $i\neq j$. Therefore, at least three distinct primes divide $|G|$.
\end{proof}
\begin{theorem}
Let $G$ be a finite nilpotent group. Then, $\Gamma(G)$ is AT-free if and only if $\lvert \pi(G)\rvert<3$.
\end{theorem}
\begin{proof}
Let $\lvert \pi(G)\rvert \geq 3$ and let $p_1,p_2,p_3$ are distinct primes which divide $\lvert G\rvert$. Since $G$ is nilpotent there are elements in $G$ of orders $p_1p_2, p_2p_3$ and $p_1p_3$. Select elements $u,v,w,x,y,z\in G$ such that $o(u)=p_1,o(v)=p_2,o(w)=p_3,o(x)=p_1p_2,o(z)=p_1p_3, o(y)=p_2p_3$. Then $\{x,y,z\}$ is an asteroidal triple and the required paths are $x\sim w\sim u\sim y$, $y\sim u\sim v\sim z$ and $x\sim w\sim v\sim z$.

\end{proof}
The following example shows that this characterization does not hold if $G$ is not nilpotent.
\begin{example}
The group $S_3\times \mathbb{Z}_5$ is not nilpotent and $\lvert \pi(S_3\times \mathbb{Z}_5)\rvert=3$, which does not contain an AT, whereas $D_{30}$, which is also not nilpotent and $\lvert\pi(D_{30})\rvert=3$, contains asteroidal tripples. 
\end{example}
\begin{theorem}
Let $G$ be a finite group such that $\lvert\pi(G)\rvert=3$ and let $\pi(G)=\{p_1,p_2,p_3\}$, then $\Gamma(G)$ has an asteroidal triple if and only if there exists elements in $G$ of orders $p_1p_2,p_1p_3$ and $p_2p_3$.
\end{theorem}
\begin{proof}
Let $x_1,x_2,x_3$ be an asteroidal triple. Then $\pi(x_i)\nsubseteq \pi(x_j)$ for $i,j=1,2,3$ and $i\neq j$ by part (b) of Proposition 1. Therefore, $o(x_1)=p_1^{k_1}p_2^{k_2}, o(x_2)=p_1^{k_3}p_3^{k_4}$ and $o(x_3)=p_2^{k_5}p_3^{k_6}$, where $k_1,k_2,...k_6$ are positive integers. Hence there are elements of required order.

Clearly three elements of given orders form an asteroidal triple.
\end{proof}

\begin{proposition}
Let $G$ be a finite group such that $\Gamma(G)$ is $AT$-free and $\lvert \pi(G)\rvert\geq 4$, then $Z(G)$ is trivial.
\end{proposition}
\begin{proof}
Suppose $Z(G)$ is not trivial and let $x\in Z(G)$, $x\neq e$. Without loss of generality assume that $o(x)$ is a prime, say, $p$. Since $\pi(G)\geq 4$ there are elements $y,z,w$ in $G$ of different prime orders other than $p$. Since $x$ commutes with all these elements $\{xy,xz,xw\}$ will form an $AT$ in $\Gamma(G)$.
\end{proof}
If $G$ is a group such that $\Gamma(G)$ is AT-free and $\lvert\pi(G)\rvert=3$, then $Z(G)$ need not be trivial, $S_3\times \mathbb{Z}_5$ is an example for that.

\begin{theorem}
The co-prime graph of symmetric group $S_n$ is AT-free if and only if $n\leq 7$.
 \end{theorem}
 \begin{proof}
 Since each element $x\in S_n$ we have a cycle decomposition, we can associate a partition representaion to each $x$, say $a_1^{r_1}a_2^{r_2}...a_k^{r_k}$ such that $a_1r_1+a_2r_2+...+a_kr_k=n$ and $o(x)=lcm(a_1,a_2,...a_k)$. 
 If $n=3$ and 4 then $\lvert\pi(S_n)\rvert=2$, hence by Proposition 3, $\Gamma(S_n)$ is AT-free. If $n=5$ and 6, then $\pi(S_n)=\{2,3,5\}$ and by Theorem 3.3, $\Gamma(S_n)$ is AT-free since $S_n$ has no elements of order 10 and 15. If $n=7$, eventhough $\pi(S_7)=\{2,3,5,7\}$, $S_7$ contains no elements of orders 15,14,21 and 35, so $\Gamma(S_7)$ contains no asteroidal triple.
 
 If $n\geq 8$, then the elements $x_1=(1\, 2)(3\, 4\, 5),\, x_2=(1\, 2\, 3)(4\, 5\, 6\, 7\, 8)$ and $x_2=(1\, 2)(3\, 4\, 5\, 6\, 7)$ constitute an asteroidal triple.
 \end{proof}
 \section{Conclusion and open problems}
We have completely characterized finite groups whose co-prime graphs are $C_4$-free, Claw-free, Cographs and Split graphs, and we have chararacterized nilpotent groups which are AT-free. We have been unable to solve the following problem.
\begin{problem}
Find a characterization for the finite groups whose co-prime graphs are AT-free. 
\end{problem}

The problem of finding asteroidal triples can be extended to other families of graphs defined on groups such as Power graphs, Enhanced power graphs, Commuting graphs, etc.

\section*{Acknowledgements}
The first author gratefully acknowledges the financial support of Council of Scientific and Industrial Research, India (CSIR) (Grant No-09/874(0029)/2018-EMR-I).
The authors would like to thank the DST, Government of India, for providing support to carry out the work under the scheme `FIST' (No.SR/FST /MS-I/2019/40).

\section*{Statements and Declarations}

\textbf{Conflict of interest} On behalf of all authors, the corresponding author states that there is no conflict of interest.





\bibliography{sn-bibliography}


\end{document}